\documentclass[12pt,leqno,fleqn]{amsart}  
\usepackage{tikz}
\usepackage{amsmath,amstext,amsthm,amssymb,amsxtra}
\usepackage[top=1.5in, bottom=1.5in, left=1.25in, right=1.25in]	{geometry}
\usepackage{txfonts} 
\usepackage[T1]{fontenc}
\usepackage{lmodern}

\usepackage{mathtools}
\mathtoolsset{showonlyrefs,showmanualtags}

\usepackage{hyperref} 
\hypersetup{
    colorlinks=true,       
    linkcolor=blue,          
    citecolor=magenta,        
    filecolor=magenta,      
    urlcolor=cyan           
}

\usepackage[msc-links]{amsrefs}

\theoremstyle{plain} 
\newtheorem{lemma}[equation]{Lemma} 
 
\newtheorem{theorem}[equation]{Theorem} 
\newtheorem{corollary}[equation]{Corollary} 
\newtheorem{question}[equation]{Question}
\newtheorem{priorResults}{Theorem}

\def\ZS{\ensuremath{\mathcal S}}
\def\ZM{\ensuremath{\mathfrak M}}
\def\MM{\ensuremath{\mathcal M}}
\def\ZB{\ensuremath{\mathfrak B}}

\def\ZI{\ensuremath{\mathbf 1}}

\def\ZR{\ensuremath{\mathbb R}}
\def\zR{\ensuremath{\mathcal R}}

\def\ZG{{\mathfrak G\,}}

\theoremstyle{plain}
\newtheorem{definition}[equation]{Definition} 

\theoremstyle{remark}

\numberwithin{equation}{section}

\newcommand {\e }[1]{\eqref{#1}}
\newcommand {\lem }[1]{Lemma \ref{#1}}
\newcommand {\cor }[1]{Corollary \ref{#1}}

\newcommand {\trm }[1]{Theorem \ref{#1}}

\newcommand {\df }[1]{Definition \ref{#1}}

\title[] {On logarithmic bounds of maximal sparse operators}
\author{Grigori A. Karagulyan}
\address{Faculty of Mathematics and Mechanics, Yerevan State
University, Alex Manoogian, 1, 0025, Yerevan, Armenia} 
\email{g.karagulyan@ysu.am}

\thanks{Research was partially supported by a grant from the Simons Foundation.  Part of this research was carried out at the American Institute of Mathematics, during a workshop on `Sparse Domination of Singular Integrals', October 2017.} 

\author{Michael T. Lacey}   

\address{ School of Mathematics, Georgia Institute of Technology, Atlanta GA 30332, USA}
\email {lacey@math.gatech.edu}
\thanks{Research supported in part by grant  from the US National Science Foundation, DMS-1600693 and the  Australian Research Council ARC DP160100153.}

	
\begin{document}
\begin{abstract}
Given sparse collections of measurable sets  $\ZS_k$, $k=1,2,\ldots ,N$, in a general measure space $(X,\ZM,\mu)$, let $ \Lambda_{\ZS_k}$ be the sparse operator, corresponding to $\ZS_k$. We show that the maximal sparse function $ \Lambda f = \max _{1\le k\le N} \Lambda_{\ZS_k} f $ satisfies
\begin{align*}
 &\| \Lambda  \| _{L^p(X) \mapsto L^{p,\infty}(X)} \lesssim \log N\cdot \|M_{\ZS}\|_{L^p(X) \mapsto L^{p,\infty}(X)},\,1\le p<\infty,
\\ 
&\lVert \Lambda  \rVert _{L^p(X) \mapsto L^p(X)} \lesssim (\log N)^{\max\{1,1/(p-1)\}}\cdot \|M_{\ZS}\|_{L^p(X) \mapsto L^p(X)},\,  1<p<\infty, 
\end{align*}
where $M_{\ZS}$ is the maximal function corresponding to the collection of sets $\ZS=\cup_k\ZS_k$. As a consequence, one can derive norm bounds for maximal functions formed from taking measurable selections of one-dimensional Calder\'on-Zygmund operators in the plane.  Prior results of this type had a fixed choice of Calder\'on-Zygmund operator for each direction.  

\end{abstract}

	\maketitle  
\section{Introduction} 

Let $ H _v f (x) = \int_\ZR f (x-tv) \frac{dt}t$ be the Hilbert transform performed in direction $ v$ in $ \mathbb R ^2 $. 
Here and throughout we take $ v$ to be a unit vector.  Given finite set of unit vectors $V$ define the operator
\begin{equation*}
H_Vf(x)=\max_{v\in V} |H_{v}f(x)|
\end{equation*}
It is a well known consequence of the Rademacher-Menshov theorem that we have 

\begin{priorResults}\label{t:H} For any finite set of unit vectors $V$ we have 
\begin{equation*}
\bigl\lVert  H_V  \bigr\rVert_{L^2\to L^2} \lesssim \log_+ \#V . 
\end{equation*}
\end{priorResults}
Here and below $\#V$ denotes the cardinality of $V$ and $ \log_+ n = \max \{1, \log_2  n\}$.
Many different extensions of this result have been studied.  
One of us \cite{MR2322743} showed that the norm bound is necessarily logarithmic in $ \#V $, 
in strong contrast to the classical result on the maximal function in a lacunary set of directions of Nagel, Stein and Wainger \cite{MR0466470}.  Namely, we have 
\begin{priorResults}[\cite{MR2322743}]\label{t:GK} For any finite set $ V$ of unit vectors it holds 
\begin{equation}\label{a26}
\bigl\lVert  H_V  \bigr\rVert_{L^2\to L^2} \gtrsim \sqrt { \log_+ \#V } .  
\end{equation}
\end{priorResults}

The maximal function variant   in the strong and weak-type estimates was first established by Nets Katz \cites{MR1711029,MR1681088}.    Namely, set 
\begin{equation*}
M_v f (x) = \sup _{t>0} (2t) ^{-1} \int _{-t} ^{t} \lvert  f (x-tv)\rvert\;dt ,  
\end{equation*}
for unit vectors $ v$, and for a finite set of unit directions $ V$, let $ M_V f = \max _{v\in V} M_v f $. 

\begin{priorResults}[\cites{MR1711029,MR1681088}]\label{t:katz} For any set of unit vectors $ V$, we have 
\begin{equation} \label{e:nets}
\lVert M_V\rVert _{L^2\to L^{2,\infty}} \lesssim \sqrt {\log_+ \# V }, 
\qquad 
\lVert M_V\rVert _{L^2\to L^2} \lesssim  {\log_+ \# V  }.  
\end{equation}
\end{priorResults}

Many extensions of these results have been considered, and we will cite several of these extensions. 
Herein, we prove results, which allow for much rougher examples than  singular integrals in a choice of directions. 
Let $K_a(x)$, $a\in \ZR$, be a family of Calder\'on-Zygmund kernels with uniformly bounded Fourier transforms, $\|\hat K_a \|_\infty<M$, such that $K_a(x)$ as a function in two variables $a$ and $x$ is measurable on $\ZR^2$. For a unit vector $v$ in $\ZR^2$ with a perpendicular vector $v ^{\perp}$ we consider an operator $T_v$ written by
\begin{equation}
T_vf(x)=\int_{\ZR} K_{x\cdot v^\perp}(t)f(x-tv)dt,\quad x\in \ZR^2,
\end{equation}
for compactly supported smooth functions $f$ on $\ZR^2$. Notice that on the $v$-directer lines $x\cdot v^\perp=l$  the operator $T_v$ defines one dimensional Calder\'on-Zygmund operators, and those can be different as the line varies. For a finite collection of unit vectors $V$ denote
\begin{equation}
T_Vf(x)=\max_{v\in V}|T_vf(x)|.
\end{equation}
Among the others below, as a corollary to our main result we derive the following.  

\begin{corollary}\label{c:H} If the family of Calder\'on-Zygmund kernels $K_a(x)$ satisfies the above  conditions, then for any finite collection of unit vectors $V$, we have 
\begin{align*}
&\lVert T_V  \rVert _{ L^2\to L^{2,\infty}}  \lesssim( \log_+ \lvert  V\rvert) ^{3/2} , 
\\
&\lVert T_V  \rVert _{ L^2\to L^{2} } \lesssim (\log_+ \lvert  V\rvert)^2. 
\end{align*}
\end{corollary}

No prior result we are aware of has permitted a variable choice of operator, as the line varies.  
The method of proof is by way of \emph{sparse operators}. 
Namely we use the recent pointwise domination of singular integrals by a positive operator \cites{MR3521084,150805639,MR3625108} to reduce the corollary above to a setting, where the operators are \emph{positive}. 
These positive operators, called sparse operators are `bigger than the maximal function by logarithmic terms', and so the proofs of the sparse operator bounds imply the corollary above.

\section{Sparse Operators} 
Let $(X,\ZM,\mu)$ be a measure space. Given collection of measurable sets $\ZB\subset \ZM$ defines the maximal function
\begin{equation*}
\MM_\ZB f(x)=\sup_{B\in \ZB}\langle f\rangle_B\cdot \ZI_B(x),
\end{equation*}
where $ \langle f \rangle_B = \mu  (B)^{-1} \int _{B} \lvert  f\rvert  $. 
By a \emph{sparse operator} we mean an operator 
\begin{equation*}
\Lambda_{\ZS} f(x) = \sum_{S\in \mathcal S} \langle f \rangle_S \mathbf 1_{S}(x), 
\end{equation*}
where $ \mathcal S\subset \ZM$ is a sparse collection of measurable sets, that means there is a constant $0<\gamma<1$ so that any set $S\in\ZS$ has a portion $ E_S\subset S$ with $ \mu(E_S) \geq   \gamma \mu(S)$ and those are pairwise disjoint.   

Without recalling the exact definition of a bounded Calder\'on-Zygmund operator, the main result we need from \cites{MR3521084,150805639,MR3625108}  is this. 

\begin{priorResults}\label{t:sparse} For any bounded Calder\'on-Zygmund operator $ T$, and compactly supported function $f$ on $\mathbb R^n $, there is a sparse collection $\ZS=\ZS_{T,f}$ of $n$-dimensional balls so that  
\begin{equation*}
\lvert  T f (x)\rvert \lesssim \Lambda _{\mathcal S} f  (x).   
\end{equation*}

\end{priorResults}

This inequality contains many deep results about Calder\'on-Zygmund operators, for which we refer the reader to the referenced papers.   Sparse bounds hold for other functionals of Calder\'on-Zygmund operators, like variational estimates \cite{160405506}. 
The result above  has been extended in a number of interesting ways.
Among many we could point to, the reader can consult \cites{161103808,2016arXiv161209201C,160305317,MR3531367}.  

\begin{definition}\label{d1}
	Let $(X,\ZM,\mu)$ be a measure space. A family of measurable sets $\ZB\subset \ZM$ is said to be martingale collection if for any two elements $A,B\in \ZB$ we have either 
	\begin{equation*}
	A\subset B,\quad B\subset A \text { or } A\cap B=\varnothing.
	\end{equation*} 
	We say that $\ZB$ is a finite-martingale collection if there are finite number of martingale collections 
	\begin{equation}\label{a13}
	\ZB_1,\ldots,\ZB_d
	\end{equation}
	such that for any $B\in \ZB$ there is a set $B'\in \cup_k\ZB_k$ with
	\begin{equation*}
	B\subset B',\quad \mu(B')\le C\mu(B).
	\end{equation*}
\end{definition}
It is well known that any family of balls in $\ZR^n$ forms a finite-martingale collection. Moreover, the corresponding martingale collections \e {a13} can be taken to be dyadic grids. Such dyadization is a key point in many applications of sparse operators.  

We turn to the statement of the main theorem. Let  $\ZS_k$, $k=1,2,\ldots,N$ be a finite-martingale sparse collections in a measure space $(X,\ZM,\mu)$, and suppose $\ZS=\cup_{k=1}^N\ZS_k$. The family $\ZG=\{\ZS_1,\ldots,\ZS_N\}$ defines the operator
\begin{equation}\label{a25}
\Lambda_\ZG f(x)=\max_{1\le k\le N}\Lambda _{\ZS_k}f(x).
\end{equation}

\begin{theorem}\label{t:main} With the notations above we have 
the inequalities 
\begin{align}
&\lVert \Lambda_\ZG  \rVert _{L^p \to L^{p, \infty} } \lesssim  \log_+ N \cdot \lVert \MM _{\ZS}\rVert _{L^p\to L^{p,\infty}},\,1\le p<\infty, 
\label{e:main1}\\ 
&\lVert \Lambda_\ZG  \rVert _{L^p \to L^{p} } \lesssim  (\log_+ N)^{\max\{1,1/(p-1)\}} \cdot \lVert \MM _{\ZS}\rVert _{L^p\to L^{p}},\,1<p<\infty.\label{e:main2}  
\end{align}
\end{theorem}
Here and below the notation $a\lesssim b$ will stand for the inequality $a\le c \cdot b$, where the constant $c>0$ may depend only on $p$ and on the constants from the above definitions of different type of set collections. As we said, a sparse operator is logarithmically larger than a maximal function, as indicated after \cor {c:H}. Our inequalities above match this heuristic. In fact, \cor {c:H} as well as \cor {C3} below may be analogously formulated in $\ZR^n$ for any $n\ge 2$, taking instead of parallel lines parallel hyperplanes of dimension $m< n$ and consider different $m$-dimensional Calder\'on-Zygmund operator on each hyperplane.

For a direction $ v$, and a smooth compactly supported function $ f$, we let 
\begin{equation*}
\ZS _{v} f (x) =  \Lambda _{\mathcal S (v _{\perp} \cdot x)} f (x),  
\end{equation*}
where $ v _{\perp} $ is orthogonal to $ v$, and $ y \mapsto \Lambda _{\mathcal S (y)}$ is a measurable choice of sparse operators.  
Given a finite set of unit vectors $ V$, we set $ \ZS _{V} f = \max _{v\in V} \ZS_v f $.  
 
\begin{corollary}\label{C3}
	With the notation above, for any finite set of unit vectors $V$ we have the inequalities
	\begin{align}
	&\|\ZS_V\|_{L^2\to L^{2,\infty}}\lesssim (\log_+ V)^{3/2},\label{a23}\\
	&\|\ZS_V\|_{L^2\to L^{2}}\lesssim (\log_+ V)^{2},\label{a21}\\
	&\|\ZS_V\|_{L^p\to L^p}\lesssim (\log_+ V)^{1+1/p},\quad p> 2.\label{a22}
	\end{align}
\end{corollary}
\cor {C3} immediately follows from \trm {t:main}. Indeed, there is no need to consider the measurable choice of sparse operators directly. By standard arguments, it suffices to consider a simplified discrete situation described here. For any pair of orthogonal vectors $ (v, v ^{\perp})$, let 
$ \mathcal R_v$ be the collection of dyadic rectangles in the plane, in the coordinates $ (v , v ^{\perp})$, whose lengths in the direction $ v ^{\perp}$ is one (see Fig. 1). 
\begin{figure}
	\begin{tikzpicture}[rotate=15] 
	\draw[xstep=.5cm, ystep=2cm] (-1.1,-1.1) grid (6.1,3.1);  
	\draw[thick,->] (0,0) -- (0,.5) node [left] {$ v$}; \draw[thick,->] (0,0) -- (.5,0) node [above] {$ v ^{\perp}$};
	\end{tikzpicture}
	\caption{The rectangles in $ \mathcal R_v$.}
\end{figure}
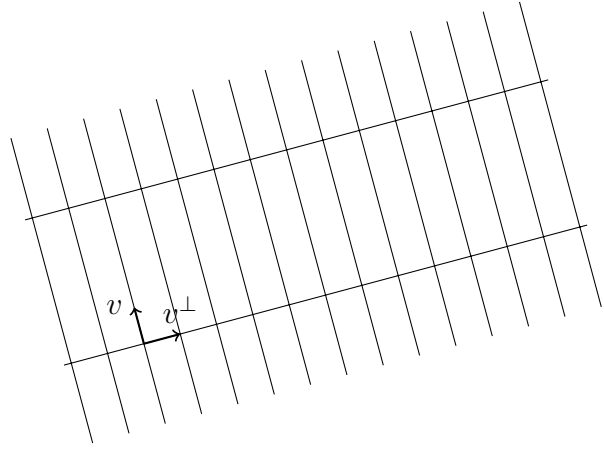
Let $ V=\{v_1,\ldots, v_N\}$ be a finite collection of unit vectors and $ \ZS _ k \subset \zR_{v_k}$, $k=1,2,\ldots, N$, be a sparse collections of rectangles. One can easily see that the operator \e {a25} generated by those collections is a discrete version of $\ZS_V$ from \e {a23}, \e {a21} and \e {a22}. On the other hand for the maximal function $\MM _{\ZS} f$ corresponding to the family of sets $\ZS=\cup_k\ZS_k$ we have the bound
\begin{equation} \label{e:MV}
\MM _{\ZS} f \le M_Vf=\max _{v\in V} \MM_{\zR_v} f ,
\end{equation}
so it satisfies to apply inequalities \eqref{e:main1} and \eqref{e:main2} combined with estimates \e {a26}. For \e {a22} we will additionally need the bound
\begin{equation*}
\|\MM _{\ZS}\|_{L^p\to L^p}\lesssim (\log_+ V)^{1/p},\quad p>2,
\end{equation*}
which is obtained from \e {a26} by the Marcinkiewicz interpolation theorem.

In light of the pointwise sparse bound in Theorem~\ref{t:sparse}, one can easily see that Corollary~\ref{c:H} in turn follows from \e{a23} and \e {a21}.  

 Since the maximal function corresponding to the $n$-dimensional canonical rectangles (with sides parallel to axes) in $\ZR^n$ is bounded on $L^p(\ZR^n)$, $1<p\le \infty$, applying the Marcinkiewicz interpolation theorem, from \e{e:main1} we can immediately deduce the following result.
\begin{corollary}\label{C2}
	If $\ZS_k$, $k=1,2,\ldots, N$, are sparse collections of canonical rectangles in $\ZR^n$, then for the maximal sparse operator \e {a25} it holds the inequality
\begin{equation}\label{a24}
\|\Lambda_\ZG\|_{L^p\to L^p}\lesssim \log_+ N,\quad 1< p<\infty,
\end{equation}
\end{corollary}
Applying the weak-$L^1$ estimate of the maximal function corresponding to $n$-dimensional balls in $\ZR^n$, from \e{e:main1} we also obtain 
\begin{corollary}\label{C4}
	If $\ZS_k$, $k=1,2,\ldots, N$, are sparse collections of balls in $\ZR^n$, then for operator \e {a25} we have
	\begin{equation}\label{a32}
	\|\Lambda_\ZG\|_{L^1\to L^{1,\infty}}\lesssim \log_+ N
	\end{equation}
\end{corollary}
Combining sparse domination \trm {t:sparse} with \cor {C2}, one can easily get
\begin{corollary}\label{C5}
	Let $T$ be a Calder\'on-Zygmund operator on $\ZR^n$. Then for any sequence of measurable functions $f_k$, $k=1,2,\ldots, N$, satisfying $|f_k(x)|\le f(x)$, $x\in \ZR^n$, it hold the inequalities
	\begin{align}
	&\left\|\sup_{1\le k\le N}Tf_k\right\|_{L^p}\lesssim \log_+N\cdot \|f\|_{L^p},\, 1< p<\infty,\label{a30}\\
	&\left\|\sup_{1\le k\le N}Tf_k\right\|_{L^{1,\infty}}\lesssim \log_+N\cdot \|f\|_{L^1}.\label{a31}
	\end{align}
\end{corollary} 
Indeed, applying \trm {t:sparse}, we get
\begin{equation*}
|Tf_k|\lesssim \Lambda _{\ZS_k} f_k\le \Lambda _{\ZS_k} f,
\end{equation*}
for some sparse collections of balls $\ZS_k$, and then the estimates in \cor {C5} can be deduced from \e {a24} and \e {a32} respectively.

\section{Proof of \trm {t:main}}
From \df {d1} it easily follows that any sparse operator, corresponding to a finite-martingale sparse collection of sets, can be dominated by a sum of bounded number of martingale sparse operators. So we can consider only martingale collections $\ZS_k$ in \trm {t:main}. 

The basic key to the proofs are the following properties of a sparse collection.
Let $\ZS$ be a martingale sparse collection. For $R\in \ZS$ denote by $\ZS_{j} (R)$ the $j$ generation of $R$. That is $\ZS_0(R)=\{R\}$ and inductively set $ \ZS _{j+1} (R)$ to be the maximal elements in 
\begin{equation}\label{a16}
\{ R'\in \ZS \;:\; R'\subset R\} \setminus \bigcup _{i=0}^{j} \ZS _{i} (R).  
\end{equation} 
Observe that for a fixed $j\ge 0$ the collection of sets
\begin{equation}
G_j(R)=\bigcup_{R'\in \ZS_{j}(R)}R', \quad R\in \ZS,
\end{equation}
is itself martingale sparse collection. Besides, from the definition of martingale sparse collection it follows that \begin{equation}\label{a17}
\mu(G_j(R))\le \gamma^j \mu(B). 
\end{equation}
This implies the exponential estimate  
\begin{equation}\label{e:exp}
\mu\Bigl\{ \sum_{\substack{S\in \mathcal S\\ S\subset R_0 }} \mathbf 1_{S} > \lambda  \Bigr\} \lesssim \lvert  R_0\rvert  \cdot  \gamma ^{\lambda}.  
\end{equation}
 
\begin{proof}[Proof of \e {e:main1}]
Take $ f \in L ^{p} (X)$, $p\ge 1$, of norm one. For a $\lambda>0$ and a small constant $ \delta>0$ we denote
\begin{align*}
&\ZS_{k,0}=\left\{R\in \ZS_k:\, \langle f \rangle_R > \frac{\delta\lambda}{\log N}  \right\},\\
&\ZS_{k,s}=\left\{R\in \ZS_k:\, \frac{\delta\lambda}{\log N}\cdot 2 ^{-s+1} \ge   \langle f \rangle _{R} >\frac{\delta\lambda}{\log N}\cdot 2 ^{-s} \right\},\, s=1,2,\ldots .
\end{align*}
Observe that for a fixed $k$ the families $\ZS_{k,s}$, $s=0,1,2,\ldots$, form a partition for 
the sparse collection $\ZS_k$. Besides, we have 
\begin{align}\label{a2}
\mu\left(\bigcup_{k=1}^N\bigcup_{R\in \ZS_{k,s}}R\right)&\le\mu\left\{\MM_{\ZS}f>\frac{\delta\lambda}{\log N}\cdot 2^{-s}  \right\}\\
&\le \left(\frac{\log N}{\delta\lambda}\right)^p\cdot 2^{sp}\cdot \lVert M _{\ZS} \rVert _{L^p\to L^{p,\infty}}^p.\\
\end{align}
Hence, using the definition of $\ZS_{k,s}$, we get
\begin{align} \label{a14}
E_\lambda  & = \{\Lambda_{\ZG}  f > \lambda \} = \bigcup _{k=1} ^N \{\Lambda _{\ZS_k} f > \lambda \} 
\\
& \subset \bigcup _{k=1} ^N \bigcup _{s\ge 0} 
\Bigl\{ \sum_{R\in \ZS _{k,s}} \langle f \rangle_R \mathbf 1_{R} > c 2 ^{-s/2} \lambda  \Bigr\} 
\\
& \subset 
\left(\bigcup _{k=1} ^N \bigcup_{R\in \ZS _{k,0}} R\right)\bigcup\left(\bigcup _{k=1} ^N \bigcup _{s\ge 1} 
\Bigl\{ \sum_{R\in \ZS _{k,s}} \langle f \rangle_R \mathbf 1_{R} > c 2 ^{-s/2} \lambda  \Bigr\}  \right)\\
& \subset 
\left(\bigcup _{k=1} ^N \bigcup_{R\in \ZS _{k,0}} R\right)\bigcup\left(\bigcup _{k=1} ^N \bigcup _{s\ge 1} 
\Bigl\{ \sum_{R\in \ZS _{k,s}}   \mathbf 1_{R} > c2 ^{s/2-1}\cdot   \frac{\log N}{\delta}  \Bigr\} \right), 
\end{align}
where $c>0$ is an absolute constant. From \e {a2} we deduce
\begin{equation}\label{a3}
\mu\left(\bigcup _{k=1} ^N \bigcup_{R\in \ZS _{k,0}} R\right)\lesssim \left(\frac{\log N}{\delta\lambda}\right)^p\lVert M _{\ZS} \rVert _{L^p\to L^{p,\infty}}^p.
\end{equation}
Applying exponential estimate \e {e:exp} and \e {a2} again, we see that 
\begin{align}\label{a4}
\mu\lvert 
\Bigl\{ \sum_{R\in \ZS _{k,s}}   \mathbf 1_{R} > c2 ^{s/2-1}\cdot   \frac{\log N}{\delta}  \Bigr\}
& \lesssim 
(\gamma ^{c/(2\delta)})^{ 2 ^{s/2}\log N } \mu\Bigl( \bigcup _{R\in \ZS _{k,s}} R \Bigr) 
\\
&\lesssim (\gamma ^{c/(2\delta)}) ^{2 ^{s/2} \log N }\cdot  2 ^{sp}\cdot \left(\frac{\log N}{\delta\lambda}\right)^p\lVert M _{\ZS} \rVert _{L^p\to L^{p,\infty}}^p \\
&\le \frac{1}{N}\cdot 2^{-s} \cdot \left(\frac{\log N}{\delta\lambda}\right)^p\lVert M _{\ZS} \rVert _{L^p\to L^{p,\infty}}^p,
\end{align}
where the last inequality is obtained by a small enough choice of $\delta$. From \e {a14}, \e {a3} and \e {a4} we immediately get
\begin{equation*}
\mu(E_\lambda)\lesssim \left(1+\sum _{k=1} ^N \sum _{s\ge 1}\frac{2^{-s}}{N} \right)\left(\frac{\log N}{\delta\lambda}\right)^p\lVert M _{\ZS} \rVert _{L^p\to L^{p,\infty}}^p\lesssim \left(\frac{\log N}{\lambda}\right)^p\lVert M _{\ZS} \rVert _{L^p\to L^{p,\infty}}^p,
\end{equation*}
that implies \e {e:main1}. 

\end{proof}

	To prove \e {e:main2} we will need a simple lemma below. Let $\ZS$ be a martingale sparse  collection with a constant $\gamma$. Attach to each $R\in \ZS$ a measurable set $G(R)\subset R$ such that $\mu(G(R))<\delta \mu(R)$, $0<\delta<1$ and suppose that $\ZS'=\{G(R):\,R\in \ZS\}$ is  itself a martingale sparse collection with the same constant $\gamma$. For $\alpha>0$ consider the sparse like operator
	\begin{equation}\label{a11}
	\Lambda_{\ZS,\ZS'}^\alpha  f(x)=\left(\sum_{R\subset \ZS}\langle f\rangle_{R}^\alpha\ZI_{G(B)}(x)\right)^{1/\alpha}.
	\end{equation}
Notice that in the case $\alpha=1$ and $G(R)=R$ it gives the ordinary sparse operator. The proof of the following lemma is based on a well-known argument. 
	\begin{lemma}\label{L1}
	The operator \e {a11} is bounded on $L^p(X)$ for $1<p<\infty$. Moreover, we have 	
	\begin{equation*}
	\|\Lambda_{S, S'}^\alpha \|_{L^p(X)\to L^p(X)}\le c\delta^{1/p}.
	\end{equation*}
where $c>0$ is a constant depended on $\alpha$ and on the constants from the above definitions.
\end{lemma}
\begin{proof}For $R\in \ZS$ we have 
	\begin{equation}\label{a15}
	\mu(G(R))\le \delta\mu(R)\le \delta\cdot \gamma^{-1}\mu(E_R),\quad \mu(G(R))\le \gamma^{-1}\mu(E_{G(R)}),
	\end{equation}
where $E_R$ and $E_{G(R)}$ denote the disjoint portions of the members of $\ZS$ and $\ZS'$ respectively.  
Suppose $\|f\|_p=1$. For some positive function $g\in L^{p/(p-\alpha)}(X)$ of norm one, we have 
	\begin{align}\label{a10}
	\|\Lambda_{\ZS,\ZS'}^\alpha (f)\|_p^\alpha&=\bigg\|\sum_{R\in \ZS_k}\langle f \rangle_{R}^\alpha\ZI_{G(R)}\bigg\|_{p/\alpha}\\
	&=\left\langle \sum_{R\in \ZS_k}\langle f \rangle_{R}^\alpha \ZI_{G(R)},g\right\rangle\\
	&=\sum_{R\in \ZS_k}\langle f \rangle_{R}^\alpha\langle g\rangle_{G(R)}\mu(G(R))\\
	&=\sum_{R\in \ZS_k}\langle f \rangle_{R}^\alpha\bigg(\mu(G(R))\bigg)^{\alpha /p}\cdot \langle g\rangle_{G(R)}\bigg(\mu(G(R))\bigg)^{(p-\alpha)/p}\\
	&\le \left(\sum_{R\in \ZS_k} \langle f \rangle_{R}^p\cdot \mu(G(R))\right)^{\alpha/p}
	\left(\sum_{R\in \ZS_k} \langle g \rangle_{G(R)}^{p/(p-\alpha)}\cdot \mu(G(R))\right)^{(p-\alpha)/p}\\
	&\le \gamma^{-1}\delta^{\alpha/p}\left(\sum_{R\in \ZS_k} \langle f \rangle_{R}^p\mu(E_R))\right)^{\alpha/p}\left(\sum_{R\in \ZS_k} \langle g \rangle_{G(R)}^{p/(p-\alpha)}\mu(E_{G(R)})\right)^{(p-\alpha)/p}\\
	&\le \gamma^{-1}\delta^{\alpha/p}\|\MM_\ZS(f)\|_p^{\alpha}\|\MM_{\ZS'}(g)\|_{p/(p-\alpha)}\\
	&\lesssim \gamma^{-1} \delta^{\alpha/p}\|f\|_p^\alpha\|g\|_{p/(p-\alpha)}=\gamma^{-1} \delta^{\alpha/p}. 
	\end{align}
In the last inequality we use the boundedness of maximal functions $\MM_\ZS$ and $\MM_{\ZS'}$ corresponding to the martingale sparse collections $\ZS$ and $\ZS'$.
\end{proof}
\begin{proof}[Proof of \e {e:main2}]
Let $ E_k$, $k=1,2,\ldots,N$ be a measurable partition of $X$.  Linearizing the supremum in the definition of $\Lambda $, we can redefine  
\begin{equation*}
\Lambda_{\ZG} f (x) = \sum_{k=1}^N \sum_{R\in \ZS_k} \langle f \rangle _{R} \mathbf 1_{E_{k,R}}(x), 
\qquad E_{k,R} = E_k\cap R . 
\end{equation*} 
Denote $\alpha=\min\{1,p-1\}\le 1$. Let $\ZS _{k, j} (R)$ be the $j$ generation of $ R\in \ZS_k$ (see the definition in \e {a16}). For a function $f\in L^p(X)$ of norm one we denote 
\begin{align}
A_j^{(1)}& =\int_X\sum_{k=1}^N \sum_{R\in \ZS_k} \langle f \rangle_R\ZI_{E_{k,R}}\left(\sum_{k=1}^N \sum_{R'\in \ZS_{k,j}(R)} \langle f \rangle_{R'}\ZI_{E_{k,R'}}\right)^\alpha\left(\Lambda_\ZG f\right)^{p-\alpha-1},\\
A_j^{(2)}&=\int_X\sum_{k=1}^N \sum_{R\in \ZS_k} \langle f \rangle_R\ZI_{E_{k,R}}\left(\sum_{k=1}^N \sum_{R':\,R\in \ZS_{k,j}(R')} \langle f \rangle_{R'}\ZI_{E_{k,R'}}\right)^\alpha\left(\Lambda_\ZG f\right)^{p-\alpha-1}.
\end{align} 
Then,  using the inequality $(\sum_kx_k)^\alpha\le \sum_kx_k^\alpha$, we get
\begin{align} \label{a18}
\|\Lambda_\ZG f\|_p^p&=\int_X\sum_{k=1}^N \sum_{R\in \ZS_k} \langle f \rangle_R\ZI_{E_{k,R}}\left(\sum_{k=1}^N \sum_{R\in \ZS_k} \langle f \rangle_R\ZI_{E_{k,R}}\right)^\alpha\cdot\left(\Lambda_\ZG f\right)^{p-\alpha-1}\\
&\le \sum_{j=0}^\infty (A_j^{(1)}+A_j^{(2)}).
\end{align}
Since $\langle f \rangle_{R'}\ZI_{E_{k,R'}}\le \MM_\ZS(f)$, for any $i=1,2$ it holds the inequality 
\begin{align}\label{a8}
A_j^{(i)}&\le \int_X(\MM_\ZS(f))^\alpha\left(\sum_{k=1}^N \sum_{R\in \ZS_k} \langle f \rangle_{R}\ZI_{E_{k,R}}\right)\cdot\left(\Lambda_\ZG f\right)^{p-\alpha-1}\\
&=\int_X(\MM_\ZS(f))^\alpha \cdot\left(\Lambda_\ZG  f\right)^{p-\alpha}\\
&\le \|\MM_\ZS\|_p^\alpha \|\Lambda_\ZG  \|_p^{p-\alpha}.
\end{align}
For a fixed $j$ and $R\in \ZS_k$ denote $G_j(R)=\cup_{R'\in \ZS_{k,j}(R)}R'$.
Observe that the family $\ZS_{k,j}=\{G_j(R):\, R\in S_k\}$ forms a martingale-system and \e {a17} implies $\mu(G_j(R))\le \gamma^{j}\mu(R)$. So the family $\ZS_{k}$ together with $\ZS_{k,j}$ satisfies the conditions of \lem {L1}. Thus we have
\begin{align}\label{a19}
A_j^{(1)}&\le \sum_{k=1}^N \int_X\sum_{R\in \ZS_k}\langle f \rangle_{R}\ZI_{G(R)} \cdot (\Lambda_\ZG  f)^\alpha\cdot \left(\Lambda_\ZG  f\right)^{p-\alpha-1}\\
&= \sum_{k=1}^N \int_X\sum_{R\in \ZS_k}\langle f \rangle_{R}\ZI_{G(R)} \left(\Lambda_\ZG  f\right)^{p-1}\\
&\le \left\|\Lambda_\ZG  \right\|_{L^p\to L^p}^{p-1}\sum_{k=1}^N\left\|\Lambda_{\ZS_{k},\ZS_{k,j}}^{1}(f)\right\|_{p}\\
&\le 2^{-cj} N\|\Lambda_\ZG  \|_{L^p\to L^p}^{p-1}\\
&\le 2^{-cj} N\|\Lambda_\ZG  \|_{L^p\to L^p}^{p-\alpha},
\end{align}
where the last inequality follows from  $\|\Lambda_\ZG  \|_{L^p\to L^p}\ge 1$. Likewise, again applying $(\sum_kx_k)^\alpha\le \sum_kx_k^\alpha$, we get
\begin{align}\label{a20}
A_j^{(2)}&\le \int_X\sum_{k=1}^N \sum_{R\in \ZS_k} \langle f \rangle_R\ZI_{E_{k,R}}\left(\sum_{k=1}^N \sum_{R':\,R\in \ZS_{k,j}(R')} \langle f \rangle_{R'}^\alpha\ZI_{E_{k,R'}}\right)\left(\Lambda_\ZG  f\right)^{p-\alpha-1}\\
&= \sum_{k=1}^N \int_X\sum_{R'\in \ZS_k} \langle f \rangle_{R'}^\alpha\ZI_{E_{k,R'}}\left(\sum_{R\in \ZS_{k,j}(R')} \langle f \rangle_{R}\ZI_{E_{k,R}}\right)\left(\Lambda_\ZG  f\right)^{p-\alpha-1}\\
&\le \sum_{k=1}^N \int_X\sum_{R'\in \ZS_k}\langle f \rangle_{R'}^\alpha\ZI_{G(R')} \cdot \Lambda_\ZG  f\cdot \left(\Lambda_\ZG  f\right)^{p-\alpha-1}\\
&= \sum_{k=1}^N \int_X(\Lambda_{\ZS_k,\ZS'_k}^\alpha(f))^\alpha \left(\Lambda_\ZG  f\right)^{p-\alpha}\\
&\le \left\|\Lambda_\ZG  \right\|_{L^p\to L^p}^{p-\alpha}\sum_{k=1}^N\left\|\Lambda_{\ZS_k,\ZS_{k,j}}^\alpha(f)\right\|_{p}^\alpha\\
&\le 2^{-cj} N\|\Lambda_\ZG  \|_{L^p\to L^p}^{p-\alpha}.
\end{align}
 Combining \eqref{a18}, \e {a8}, \e {a19} and \e {a20},  we will get	 
\begin{equation*}
\lVert \Lambda_\ZG  \rVert _{L^p\to L^p }^\alpha 
\lesssim \sum_{j=0} ^{\infty } \min \{  \lVert M _{\ZS} \rVert _{L^p\to L^p}^\alpha,N 2 ^{-cj}   \}. 
\end{equation*}
Since $\lVert M _{\ZS} \rVert _{L^p\to L^p}\ge 1$, for an appropriate choice of a constant $c'>0$ we obtain
\begin{align*}
\lVert \Lambda_\ZG  \rVert _{L^p\to L^p }^\alpha &\le  c'\log N \lVert M _{\ZS} \rVert _{L^p\to L^p}^\alpha+\sum_{j=c'\log N}^\infty
N 2 ^{-cj} \\
&\le c'\log N \lVert M _{\ZS} \rVert _{L^p\to L^p}^\alpha+1\\
&\lesssim \log N \lVert M _{\ZS} \rVert _{L^p\to L^p}^\alpha.
\end{align*}
Taking into account the definition of $\alpha$, this completes the proof of \eqref{e:main2}.

\end{proof}

\section{Extensions} 

The logarithmic gains in the main theorem are sharp, in general. Indeed, it is enough to show the optimality of logarithm in \e {a24}. The function  $ f$ is taken to be identically one on a large cube $Q\subset \ZR^n$. 
For each $k=1,2,\ldots, N$, it is very easy to construct a sparse operator $\Lambda_{\ZS_k}$ based on a sparse collection of cubes $\ZS_k$ so that 
$ \{ \Lambda_{\ZS_k} f > c \log N \}  $ will have measure at least $ \lvert  Q\rvert/N  $.  These sets can be made to be essentially statistically independent, so that one sees that the logarithmic bound is sharp in \eqref{a24}. A careful examination of the same argument can show also the sharpness of the estimates \e {a30} and \e {a31}, in general.

The papers \cites{MR3145928,MR2680067} prove a variety of results for $ T_V$ defined as a maximum of a fixed Hormander-Mihklin multiplier computed in directions $ v\in V$.  Their estimates are slightly better than ours in \cor {c:H}.  This raises two questions:

\begin{question}
	First, if one fixes the specific sparse operator computed in every direction, can bounds be proved that match those of say \cite{MR3145928}?   
\end{question}

 This paper \cite{MR3145928} proves results for the maximal truncations of the Hilbert transform computed in different directions. Again, their bounds are better than ours.

\begin{question}
	Can one formulate a maximal sparse operator which is less general than ours,
but still general enough to capture these results for maximal truncations of the Hilbert transform?  
\end{question}

Recent papers \cites{MR3047650,MR3432267} have established variants of these results in higher dimensions. 
Other papers \cites{2017arXiv170607111D,2017arXiv170402918D} consider certain Lipschitz versions.  
It would be interesting to study the analogous questions for both themes.

\bibliographystyle{plain}


\end{document}